\documentclass[12pt]{amsart}

\usepackage{amsmath,amsfonts,amssymb,mathabx,latexsym,mathtools}
\usepackage{enumitem}
\usepackage[usenames, dvipsnames]{xcolor}
\usepackage{hyperref}
\usepackage{tikz}
\usetikzlibrary{intersections}
\usepackage[backend=bibtex]{biblatex}
\usepackage[capitalize,noabbrev]{cleveref}
\newtheorem{theorem}{Theorem}[section]
\newtheorem{lemma}[theorem]{Lemma}
\newtheorem{proposition}[theorem]{Proposition}
\numberwithin{equation}{section}

\newcommand{\Dfn}[1]{\emph{\color{blue}#1}} % to highlight and collect definitions
\newcommand{\FindStat}[1]{%
  \ifx&#1&%
  \url{www.findstat.org}% #1 is empty
  \else \url{www.findstat.org/#1}% #1 is nonempty
  \fi}%

\newcommand{\Sym}{\ensuremath{\mathfrak{S}}}

\newcommand{\runsort}{runsort}
\newcommand{\reverse}{reverse}
\newcommand{\invinv}{\widecheck{\mathrm{Inv}}}
\mathchardef\mhyphen="2D
\newcommand{\ascocc}{13\mhyphen 2}
\DeclareMathSymbol{\mstar}{\mathalpha}{symbols}{"03}
\newcommand{\desview}{31\mstar 2}
\newcommand{\dv}{d} % number of 13's that make v a descent view

\newcommand{\RUNS}{R} % set of run tops
\newcommand{\RT}{RT} % set of run tops
\newcommand{\RB}{RB} % set of run bottoms
\DeclareMathOperator{\rt}{rt} % the run top associated with a value
\title{An equidistribution involving invisible inversions}
\author[M. Coopman]{Michael Coopman}
\address[MC]{Department of Mathematics, University of Florida, Gainesville, FL 32601}
\email{m.coopman@ufl.edu}

\author[M. Rubey]{Martin Rubey}
\address[MR]{Fakult\"at f\"ur Mathematik and Geoinformation, TU Wien, Austria}
\email{martin.rubey@tuwien.ac.at}
\thanks{Martin Rubey was supported by the Austrian Science Fund (FWF): P 29275.}

\begin{document}
\begin{abstract}
  We provide two explicit bijections demonstrating that, among
  permutations, the number of invisible inversions is equidistributed
  with the number of occurrences of the vincular pattern $\ascocc$
  after sorting the set of runs.
\end{abstract}

\maketitle

%%%%%%%%%%%%%%%%%%%%%%%%%%%%%%%%%%%%%%%%%%%%%%%%%%%%%%%%%%%%%%%%
%
\section{Introduction}
The set of \Dfn{visible inversions}, a certain subset of the usual inversion set of a permutation, was introduced by Z.~Hamaker, E.~Marberg and P.~Pawlowski in~\cite{MR4015854} to describe the rank function of the restriction of the Bruhat order to involutions.

In this note we are concerned with the complementary set of
\Dfn{invisible inversions}\footnote{\FindStat{St001727}}
$\invinv(\pi)$ of a permutation $\pi$.  This is the set of pairs
$(i, j)$ such that $i < j$ and $\pi(i) > \pi(j) > i$.  Querying the
database of combinatorial statistics~\FindStat{} suggested that the
number of invisible inversions is equidistributed with the number of
occurrences of the vincular pattern $\ascocc$ after sorting the set
of runs.  Our goal is to prove a refinement of this guess.  We note
that the map sorting the set of runs of a permutation appears to have
some remarkable properties, as recently demonstrated by
P.~Alexandersson and O.~Nabawanda in~\cite{arXiv:2104.04220}.

An \Dfn{ascending run} of $\pi$ is a maximal consecutive subsequence
of $\pi$ in its one-line notation that is increasing.  Let
\Dfn{$\runsort(\pi)$} be the permutation obtained by rearranging the
set of ascending runs in increasing order of their minimal
elements\footnote{\FindStat{Mp00223}}.  Furthermore, let
\Dfn{$\ascocc(\pi)$}\footnote{\FindStat{St000356}} be the set of
pairs $(i,j)$ such that $i < j$ and $\pi(i) < \pi(j) < \pi(i+1)$.
Then, it turns out that
\[
\sum_{\pi\in\Sym_n} q^{|\invinv(\pi)|} %
= \sum_{\pi\in\Sym_n} q^{|\ascocc(\runsort(\pi))|}.
\]

In fact, this result can be refined significantly.  To state the refinement, let \Dfn{$\desview(\pi)$} be the set of pairs $(i,j)$ such that $\pi(i) > \pi(j) > \pi(i+1)$ and the minimal element of the descending run containing $\pi(i)$ and $\pi(i+1)$ is smaller than the minimal element of the descending run containing $\pi(j)$.  In this case we call $\pi(j)$ a \Dfn{descent view}.

Thus, $|\desview(\pi)| = |\ascocc(\runsort(\reverse(\pi)))|$.  Furthermore, recall that a \Dfn{global ascent} of a permutation $\pi$ is an index $m$ such that $\pi$ restricted to $[m]$ is a permutation.  Finally, a \Dfn{left-to-right maximum} of $\pi$ is a value $v=\pi(j)$ such that $\pi(i) < v$ for all $i < j$.
\begin{theorem}
  \label{t:main}
  There is an explicit bijection $\chi:\Sym_n\to\Sym_n$ such that, setting $\sigma=\chi(\pi)$,
  \begin{itemize}
  \item the multiset $\{\{\pi(j) \mid (i,j)\in\desview(\pi) \}\}$ of descent views is the multiset $\{\{\sigma(j) \mid (i,j)\in\invinv(\sigma) \}\}$ of invisible inversion bottoms,
  \item the set of maximal elements of the descending runs of $\pi$ is the set of positions of the weak deficiencies of $\sigma$,
  \item the set of minimal elements of the descending runs of $\pi$ is the set of values of the weak deficiencies of $\sigma$,
  \item the sets of global ascents of $\pi$ and $\sigma$ are the same, and
  \item the set of left-to-right maxima of $\pi$ is the set of maximal elements in the cycles of $\sigma$.
  \end{itemize}
\end{theorem}
We will show in \cref{sec:exc} that $\chi(\pi)$ is uniquely
determined on its set of exceedances by the requirements of
\cref{t:main}.  We will then present two slightly different ways to
define $\chi(\pi)$ on the set of weak deficiences, the first variant,
\FindStat{Mp00235}, in \cref{sec:first-variant}, and the second
variant, \FindStat{Mp00237}, in \cref{sec:second-variant}.

We remark that, for both bijections, the preimage of the set of involutions is the set of permutations whose descending runs all have length at most two and increasing maximal elements.  Restricting $\chi$ to this set we obtain the classical (inverse) fundamental transformation\footnote{\FindStat{Mp00086}} due to R\'enyi, and Foata and  Sch\"utzenberger.

It may be interesting to explore whether further bijections satisfying some of the
requirements of \cref{t:main} can be found.  One might speculate that encoding
permutations in terms of labelled Motzkin paths with two types of
horizontal steps, using the Foata-Zeilberger bijection or one of its
variants, could be helpful.  Doing so, the exceedances of $\chi(\pi)$
and their values would carry the same information as the Motzkin path
itself and its labels on the up steps and the horizontal steps of one
kind.

One variant of the Foata-Zeilberger bijection, between involutions
and Motzkin paths with only one kind of horizontal steps and labels
only on the down steps, was in fact used by Z.~Hamaker and the first
author~\cite{arxiv:2106.06021} to give a visual interpretation of the
number of visible inversions in involutions.

\section{Definitions}
A \Dfn{(descending) run} of $\pi$ is a maximal consecutive decreasing
subsequence of $\pi$ in its one-line notation.  From now on, we will only consider descending runs.  We denote the set of
runs of $\pi$ with $R(\pi)$. A \Dfn{run top/run bottom} is the
largest/smallest element of a run.

An \Dfn{exceedance/weak deficiency} of $\pi$ is a position $i\in [n]$ such that $\pi(i)$ is strictly greater/weakly less than $i$.
An \Dfn{inversion} of $\pi$ is a pair $(i,j)$ such that $i < j$ and $\pi(i) > \pi(j)$.  We call $\pi(j)$ the \Dfn{inversion bottom}.
An \Dfn{invisible inversion} additionally satisfies $i < \pi(j)$.

\section{The exceedances of \texorpdfstring{$\chi(\pi)$}{chi(pi)}}
\label{sec:exc}
We define the bijection $\chi:\Sym_n\to\Sym_n$ in two steps.  First we describe the exceedances of $\sigma = \chi(\pi)$.  The following is our key lemma.

\begin{lemma}\label{lem:sigma_exc}
  Let $\RB$ and $\RT$ be the set of run bottoms, respectively run
  tops, of $\pi$.  There exists a unique bijection
  $\sigma_e = \chi_e(\pi): [n]\setminus\RT\to[n]\setminus\RB$ whose multiset of
  inversion bottoms is the multiset of descent views of $\pi$ which
  are not run bottoms.  Moreover, $\sigma_e(i) > i$ for all
  $i\in [n]\setminus\RT$.
\end{lemma}
\begin{proof}
  Let $\dv_{\pi}(v)$ be the number of $\pi(i)$'s that make $v$ a descent view, that is, the multiplicity of $v$ in the multiset of descent views.
  Let $v\in[n]\setminus\RB$.  Thus,  $\dv_\pi(v)$ is the number
  of positions $i$ such that $\pi(i+1) < v < \pi(i)$ and the run
  containing $\pi(i+1)$ and $\pi(i)$ has a smaller run bottom than
  the run containing $v$.  Therefore, if $\RUNS$ is the set of runs
  of $\pi$ and $b$ is the run bottom of the run containing $v$,
  \[
    \dv_\pi(v) = \#\{ r\in\RUNS \mid \min(r) < v < \max(r) \text{ and } \min(r) < b \}.
  \]
  We now define $\sigma_e$ iteratively.  Let $S=[n]\setminus\RT$.
  For $v\in[n]\setminus\RB$, beginning with the smallest, we let $\sigma_e^{-1}(v)$ be the
  $(\dv_\pi(v)+1)$-st element in $S$ and then remove this element
  from $S$.

  When defining $\sigma_e^{-1}(v)$, there are
  $|[n]\setminus\RB| - \#\{w \in [n]\setminus\RB \mid w < v\}$ elements in
  $S$, which is at least $\dv_\pi(v)+1$ because
  \begin{multline*}
    \dv_\pi(v) + 1 \leq \#\{ w \in [n]\setminus\RB \mid w > v\} + 1 \\
    = |[n]\setminus\RB| - \#\{ w \in [n]\setminus\RB \mid w < v\},
  \end{multline*}
  which ensures that $\sigma_e$ is well-defined.

  Uniqueness follows from the fact that $v$ is the inversion bottom
  of precisely those $\dv_\pi(v)$ elements whose preimage is in
  what remains of $S$ but is smaller than $\sigma_e^{-1}(v)$.

  It remains to show that
  $\sigma_e^{-1}(v) < v$.  Initially, there are
  \[
    v-1-\#\{ w\in\RT \mid w < v \}
  \]
  elements in $S$ that are strictly less than $v$.  When defining
  $\sigma_e^{-1}(v)$, we can assume by induction that only elements
  strictly less than $v$ were removed from $S$.  Moreover, the number
  of removed elements equals $\#\{ w\in[n]\setminus\RB\mid w < v\}$.
  Thus, the number of elements strictly less than $v$ in what remains
  of $S$ is
  \[
    v-1-\#\{ w\in\RT \mid w < v \} - \#\{ w\in[n]\setminus\RB\mid w < v\}.
  \]
  It remains to show that this is larger than $\dv_\pi(v)$, which
  follows from these two observations:
  \begin{multline*}
  v-1-\#\{ w\in[n]\setminus\RB \mid w < v\}%
    =\#\{ w\in\RB \mid w < v \}, \qquad \quad%Character spacing fix.
  \end{multline*}
  and
  \begin{multline*}
    \#\{ w\in\RB \mid w < v \} - \#\{ w\in\RT \mid w < v \} \\
    = \#\{ r\in\RUNS \mid \min(r) < v \leq \max(r)\} > \dv_\pi(v).
  \end{multline*}
\end{proof}

\begin{lemma}\label{lem:runs}
  Let $\pi_1$ and $\pi_2$ be two permutations.  Then
  \[
  \chi_e(\pi_1) = \chi_e(\pi_2) \Leftrightarrow R(\pi_1) = R(\pi_2).
  \]
\end{lemma}
\begin{proof}
  If the sets of runs of $\pi_1$ and $\pi_2$ coincide we have
  $\chi_e(\pi_1) = \chi_e(\pi_2)$, because the multiset of descent
  views of a permutation only depends on its set of runs.

  Conversely, $\chi_e(\pi_1) = \chi_e(\pi_2)$ immediately implies
  that the sets of run bottoms and run tops of $\pi_1$ and $\pi_2$
  are the same.  We show how to reconstruct the set of runs of a
  permutation $\pi$ from $\chi_e(\pi)$.  Let us first determine the
  matching between run bottoms and run tops.

  Let $t$ be a run top of $\pi$, beginning with the smallest.
  Suppose that $t$ is not a run bottom.  In this case $\dv_\pi(t)$ is
  the multiplicity of $t$ in the multiset of inversion bottoms of
  $\chi_e(\pi)$.  Then, as $t$ completes $\dv_{\pi}(t)$ descent
  views, $t$ must be the run top for the $(\dv_\pi(t)+1)$-st smallest
  run bottom which is not yet matched with a run top.  If $t$ is a
  run bottom, $\dv_{\pi_1}(t)$ is not determined by $\chi_e(\pi_1)$.
  However, in this case it is clearly a singleton run in both $\pi_1$
  and $\pi_2$.

  Finally, suppose the $v$ is neither a run bottom nor a run
  top. Then, among the runs with smaller run bottom and larger run
  top, $v$ belongs to the one with $(\dv_\pi(v)+1)$-st smallest run
  bottom.
\end{proof}

\begin{lemma}\label{lem:multiset}
  Let $\pi$ be a permutation and let $\sigma$ be any permutation that agrees with $\sigma_e(\pi)$ on the set of exceedances and has no other exceedances.  Then the multiset of invisible inversion bottoms of $\sigma$ equals the multiset of descent views of $\pi$.
\end{lemma}
\begin{proof}
  Let $\RB$ be the set of run bottoms of $\pi$.  If $v\in [n]\setminus\RB$, then by \cref{lem:sigma_exc} $v$ is the value of an exceedance of $\sigma$.  Therefore, $v$ is an invisible inversion bottom if and only if it is an inversion bottom.  Thus, the multiplicity of $v$ in the multiset of invisible inversion bottoms equals the multiplicity of $v$ in the multiset of descent views of $\pi$.

  Suppose now that $v\in\RB$.  We partition the set $[v-1]$ in two ways.  For any element $w$ let $\rt(w)$ be the run top in the decreasing run containing $w$. Then the first partition of $[v-1]$ is as follows:
  \begin{align*}
    A &= \{w\in[v-1]\mid w\not\in\RB\}\\
    B &= \{w\in[v-1]\mid w\in\RB\text{ and }\rt(w) \leq v\}\\
    C &= \{w\in[v-1]\mid w\in\RB\text{ and }\rt(w) > v\}.
  \end{align*}
  Note that $\#C= \dv_\pi(v)$.
  For the second partition of $[v-1]$, let
  \begin{align*}
    F &= \{i\in[v-1]\mid \sigma(i)\leq i\}\\
    G &= \{i\in[v-1]\mid \sigma(i) > i \text{ and }\sigma(i)\leq v\}\\
    H &= \{i\in[v-1]\mid \sigma(i) > i \text{ and }\sigma(i) > v\}.
  \end{align*}
  Note that $\#H$ is the multiplicity of $v$ in the multiset of
  invisible inversions bottoms, because with
  $j=\sigma^{-1}(v) \geq v$ we have $i < j$ and
  $i < v=\sigma(j) < \sigma(i)$ for any $i\in H$.  Moreover,
  $\# F = \# B$, since every element $i$ of $F$ is a run top of
  $\pi$, and mapping $i$ to the corresponding run bottom yields a
  bijection.  Finally, $\# G = \# A$, with $\sigma$ restricted to $G$
  serving as a bijection: because any $i\in G$ is an exceedance of
  $\sigma$ we find that $\sigma(i)$ is not a run bottom of $\pi$ and
  therefore $\sigma(i) \neq v$.

  We conclude that $\# C = \# H$, which is what we wanted to prove.
\end{proof}

\begin{proposition}
  \label{p:global_ascents}
  Let $\pi$ be a permutation and let $\sigma$ be any permutation that
  agrees with $\chi_e(\pi)$ on the set of exceedances and that has
  no other exceedances.  Then the sets of global ascents of $\pi$ and
  $\sigma$ are the same.
\end{proposition}
\begin{proof}
  If $\pi$ has a global ascent at index $m$,
  $\overline\pi=(\pi(1),\dots,\pi(m))$ is a permutation of $[m]$.
  Moreover, any pair $(i, j) \in \desview(\pi)$ has either both $i$
  and $j$ in $[m]$, or both in the complement of $[m]$.  The same is
  true for pairs $(i, j) \in \invinv(\pi)$.

  Therefore, the restriction of $\chi_e(\pi)$ to $[m]$ is the same
  as $\chi_e(\overline\pi)$.  Since the remaining values of
  $\sigma$ are weak deficiencies, $\sigma$ also has a global ascent at
  index $m$.

  Conversely, suppose that $\sigma$ has a global ascent at index $m$.
  Since $\sigma$ agrees with $\chi_e(\pi)$ on the set of exceedances and has no other exceedances, it maps run tops of $\pi$ to run bottoms of $\pi$.  Because it has a global ascent at $m$, any run top larger than $m$ is mapped to a run bottom larger than $m$.  Therefore, the number of run tops of $\pi$ smaller than $m$ must be equal to the number of run bottoms of $\pi$ smaller than $m$.
  Thus, any run of $\pi$ has either both its run bottom and its run top, and therefore all of its elements contained in $[m]$ or all of its elements contained in the complement of $[m]$.
  Finally, any ordering of the runs of $\pi$ must be such that the runs contained in $[m]$ must come first.  This in turn implies that $\pi$ has a global ascent.
\end{proof}

\begin{figure}
  \centering
  \begin{tikzpicture}[baseline=(A), scale=0.4]
    \node (A) at (1, 1) {};
    \draw (1, 1) grid (10, 10);
    \node at (1.5,5.5) {$\bullet$};
    \node at (2.5,2.5) {$\bullet$};
    \node at (3.5,9.5) {$\bullet$};
    \node at (4.5,6.5) {$\bullet$};
    \node at (5.5,8.5) {$\bullet$};
    \node at (6.5,7.5) {$\bullet$};
    \node at (7.5,3.5) {$\bullet$};
    \node at (8.5,1.5) {$\bullet$};
    \node at (9.5,4.5) {$\bullet$};
    \draw[very thick] (3,1) -- (3,10);
    \draw[very thick] (5,1) -- (5,10);
    \draw[very thick] (9,1) -- (9,10);
  \end{tikzpicture}
  \begin{tikzpicture}[baseline=(A), scale=0.4]
    \node (A) at (1, 1) {};
    \draw (1, 1) grid (10, 10);
    \node at (1.5,8.5) {$\bullet$};
    \coordinate (P2) at (2.5,7.5); \node at (P2) {$\bullet$};
    \coordinate (P3) at (3.5,3.5); \node at (P3) {$\bullet$};
    \coordinate (P4) at (4.5,1.5); \node at (P4) {$\bullet$};
    \coordinate (P5) at (5.5,5.5); \node at (P5) {$\bullet$};
    \coordinate (P6) at (6.5,2.5); \node at (P6) {$\bullet$};
    \coordinate (P7) at (7.5,4.5); \node at (P7) {$\bullet$};
    \coordinate (P9) at (9.5,6.5); \node at (P9) {$\bullet$};
    \node at (8.5,9.5) {$\bullet$};
    \draw[dotted] (P2) -- (P3);
    \draw[dotted] (P3) -- (P4);
    \draw[dotted] (P5) -- (P6);
    \draw[very thick] (5,1) -- (5,10);
    \draw[very thick] (7,1) -- (7,10);
    \draw[very thick] (8,1) -- (8,10);
    \coordinate (A5) at (0.5,5.5);
    \coordinate (A6) at (0.5,2.5);
    \coordinate (A7) at (0.5,4.5);
    \coordinate (A9) at (0.5,6.5);
    \draw[dashed] (intersection of  A9--P9 and P2--P3) -- (P9);
    \draw[dashed] (intersection of  A5--P5 and P2--P3) -- (P5);
    \draw[dashed] (intersection of  A6--P6 and P3--P4) -- (P6);
    \draw[dashed] ([yshift=-2pt]intersection of A7--P7 and P2--P3) -- ([yshift=-2pt]P7);
    \draw[dashed] ([yshift=2pt]intersection of A7--P7 and P5--P6) -- ([yshift=2pt]P7);
  \end{tikzpicture}
  \begin{tikzpicture}[baseline=(A), scale=0.4]
    \node (A) at (1, 1) {};
    \draw (1, 1) grid (10, 10);
    \draw (1, 1) -- (10, 10);
    \coordinate (P1) at (1.5,3.5); \node at (P1) {$\bullet$};
    \coordinate (P2) at (2.5,7.5); \node at (P2) {$\bullet$};
    \coordinate (P3) at (3.5,5.5); \node at (P3) {$\bullet$};
    \coordinate (P6) at (6.5,8.5); \node at (P6) {$\bullet$};
    \coordinate (P7) at (7.5,9.5); \node at (P7) {$\bullet$};
    \coordinate (A1) at (1.5,1.5); \draw[dotted] (A1) -- (P1);
    \coordinate (A2) at (2.5,2.5); \draw[dotted] (A2) -- (P2);
    \coordinate (A3) at (3.5,3.5); \draw[dotted] (A3) -- (P3);
    \coordinate (A6) at (6.5,6.5); \draw[dotted] (A6) -- (P6);
    \coordinate (A7) at (7.5,7.5); \draw[dotted] (A7) -- (P7);
    \draw[dashed] (1.5,2.5) -- (9.5,2.5);
    \draw[dashed] ([yshift=-2pt]2.5,4.5) -- ([yshift=-2pt]9.5,4.5);
    \draw[dashed] ([yshift=2pt]3.5,4.5) -- ([yshift=2pt]9.5,4.5);
    \draw[dashed] (2.5,5.5) -- (P3);
    \draw[dashed] (2.5,6.5) -- (9.5,6.5);
  \end{tikzpicture}
  \caption{Descent views of $\pi = 52\,96\,8731\,4$ and
    $8731\,52\,4\,96$, and invisible inversions of $\sigma_e = 375??89??$.}
  \label{fig:des-views-inv-inv}
\end{figure}
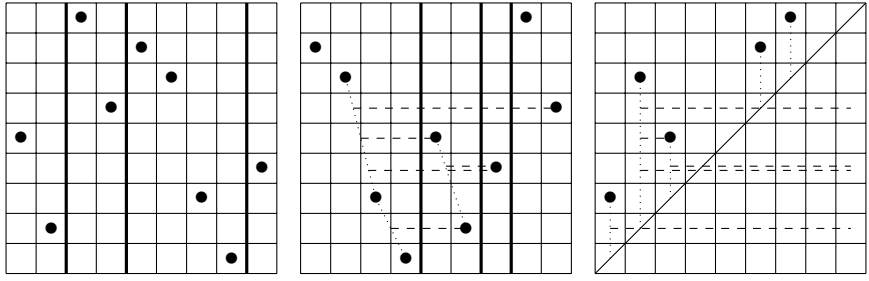

Let us conclude this section with an example.  Let
$\pi = 52\;96\;8731\;4$.  In the left most diagram of
\cref{fig:des-views-inv-inv} there is, for each $i$, a black dots in
the $i$-th column from the left and the $\pi(i)$-th row from the
bottom.  The descending runs of $\pi$ are separated by heavier
vertical lines.

The permutation $\overline\pi$ illustrated in the middle diagram is
obtained by sorting the set of descending runs of $\pi$ such that the
sequence of run bottoms increases.  Thus, the multisets
$\{\{\pi(j) \mid (i,j)\in\desview(\pi) \}\}$ and
$\{\{\overline\pi(j) \mid (i,j)\in\desview(\overline\pi)\}\}$ are the
same, but the multiset for $\overline\pi$ is easier to visualize.  In the
middle diagram, the descent views are indicated by horizontal dashed
lines, emanating from the dotted lines, which indicate the descents,
towards the right.

For both $\pi$ and $\overline\pi$ we have
$[n]\setminus\RB = (3,5,7,8,9)$ and the multiset of descent views is
$\{\{2,4,4,5,6\}\}$.  The only descent view which is not a run bottom
is $5$, and we have
$(\dv_{\pi}(v))_{v\in[n]\setminus\RB} = (0,1,0,0,0)$.

The partial permutation $\sigma_e = 375??89??$ is illustrated in the
right most diagram.  Additionally, the dashed horizontal lines
emanating from the vertical dotted lines indicate the invisible
inversions.  The only invisible inversion $(i, j)$ such that
$\sigma(j)$ is not a run bottom has $\sigma(i) = 5$.

There are twelve permutations with the same set of runs as $\pi$, and
twelve permutations whose restriction to $[n]\setminus\RT$ equals
$\sigma_e$ and which do not have any further exceedances.

%%%%%%%%%%%%%%%%%%%%%%%%%%%%%%%%%%%%%%%%%%%%%%%%%%%%%%%%%%%%%%%%%%%%%%

\section{The weak deficiencies of \texorpdfstring{$\chi(\pi)$}{chi(pi)}}

\Cref{lem:sigma_exc} shows that the values of $\sigma$ on
$[n]\setminus\RT$ are determined by the first three requirements in
\cref{t:main}.  This is not the case for the values on the run tops
of $\pi$, even when imposing the final requirement stated in
\cref{t:main}.

In this section we provide two different ways to complete the
definition of $\sigma$.  Note that, if we ignore the final
requirement from \cref{t:main}, it suffices to encode the order of
the runs of $\pi$ in the bijection $\sigma_d: \RT\to\RB$.

In both variants we will make use of the following construction: let
$f$ be a partial permutation, that is an injective map $S\to [n]$
defined on a subset $S$ of $[n]$.  Then the \Dfn{iterated preimage}
$f^*(c)$ of an element $c\in [n]$ which is not contained in a cycle
of $f$ is the element $s\in S$ such that $f^k(s) = c$, with $k\geq 0$
maximal.  In particular, if $c$ is not in the image of $f$, then
$f^*(c) = c$.

We define the values of $\sigma$ on the set of run tops iteratively,
beginning with the smallest run top and removing the corresponding
runs from $\pi$.  Thus, throughout the algorithm, we consider
$\sigma$ as a partial permutation, which initially equals $\sigma_e$.
Note that $\sigma_e$ does not contain any cycles.

The first three steps of both variants are identical:
\begin{enumerate}
\item Let $t$ be the smallest run top in $\pi$.
\item Remove the run containing $t$ from $\pi$.
\item\label{finish} If $t$ was the first element of $\pi$, set
  $\sigma(t) = \sigma^*(t)$.
\end{enumerate}

In the following two sections, we provide two different ways of
specifying $\sigma(t)$ if there is a run bottom to the left of $t$.
In both cases, we set $\sigma(t) = \sigma^*(t')$ for some run top
$t' > t$ satisfying $\sigma^*(t') < t$.  Thus, throughout the
algorithm the functional digraph of $\sigma$ is a collection of paths
beginning with a run bottom and ending with a run top, which is also
the largest element of the path, together with a collection of
cycles.  In each iteration a run top is connected to a smaller run
bottom.  Assuming this, we can already conclude that the bijection
satisfies the final requirement of \cref{t:main}.
\begin{lemma}
  A cycle of $\sigma$ is completed in an iteration of the algorithm
  if and only if the smallest run top $t$ is the first element of
  $\pi$.  In this case, $t$ is the maximal element of the cycle.

  Therefore, the left-to-right maxima of $\pi$ are the maximal
  elements of the cycles of $\sigma$.
\end{lemma}
\begin{proof}
  If the minimal run top $t$ is the first element of $\pi$, and
  therefore a left-to-right maximum, we complete a cycle of $\sigma$
  by setting $\sigma(t) = \sigma^*(t)$.  This cycle has maximal
  element $t$ because only run tops $t' < t$ have been assigned an
  image yet, and, by assumption, all run tops are weak deficiencies.

  Otherwise, we set $\sigma(t) = \sigma^*(t')$ for some run top
  $t' > t$.  Since the (partial) map $\sigma$ is injective,
  $\sigma^*(t')$ must be different from $\sigma^*(t)$ before this
  assignment, which means that no cycle is completed.
\end{proof}

\subsection{First variant}
\label{sec:first-variant}
In the following we will use an auxiliary map $\tau$, which we now
define.  Let $\rho:\RB\to\RT$ be the bijection that maps each run
bottom of $\pi$ to the corresponding run top in the same run.  Then,
given $\sigma$, we set $\tau = \sigma^* \circ \rho$, restricted to
the set of run bottoms remaining in $\pi$.  Note that $\tau$ has to
be recomputed whenever another value of $\sigma$ is determined.

For example, if $\pi = 52\;96\;8731\;4$, we have
\[
  \rho = \begin{smallmatrix}
    1 & 2 & 4 & 6\\ 8 & 5 & 4 & 9
  \end{smallmatrix}
  \text{ and }
  \sigma_e^* = \begin{smallmatrix}
    8 & 5 & 4 & 9\\ 6 & 1 & 4 & 2
  \end{smallmatrix}
\]
and therefore
\[
  \tau = \begin{smallmatrix}
    1 & 2 & 4 & 6\\ 6 & 1 & 4 & 2
  \end{smallmatrix}.
\]

We now specify the final step of the algorithm as follows:
\begin{enumerate}
  \setcounter{enumi}{(3)}
\item \label{continue-v1} Otherwise, let $a$ be the run bottom left
  of $t$, and let $k \geq 1$ be minimal such that $\tau^{k}(a) < t$.
  Set $\sigma(t) = \tau^{k}(a)$.
\end{enumerate}

Let us illustrate this procedure by resuming our example.  Initially,
we have
\[
  \pi = 52\;96\;8731\;4\qquad
  \sigma =
  \begin{smallmatrix}
    1 & 2 & 3 & 4 & 5 & 6 & 7 & 8 & 9\\
    3 & 7 & 5 &   &   & 8 & 9 &   &
  \end{smallmatrix}\qquad
  \tau = \begin{smallmatrix}
    1 & 2 & 4 & 6\\
    6 & 1 & 4 & 2
  \end{smallmatrix}.
\]
The smallest run top is $t=4$.  There is a run to the left of $t$,
with run bottom $a = 1$.  Since $\tau(a) = 6 > 4$, and
$\tau^2(a) = 2 < 4$, we set $\sigma(4)=2$.

Removing the run $4$ from $\pi$ and taking into account that now
$\tau(6)=\sigma^*(9)=4$ we obtain
\[
  \pi = 52\;96\;8731\qquad
  \sigma =
  \begin{smallmatrix}
    1 & 2 & 3 & 4 & 5 & 6 & 7 & 8 & 9\\
    3 & 7 & 5 & 2 &   & 8 & 9 &   &
  \end{smallmatrix}\qquad
  \tau = \begin{smallmatrix}
    1 & 2 & 6\\
    6 & 1 & 4
  \end{smallmatrix}.
\]
The smallest run top is now $t=5$.  Since $t$ is the first element of
$\pi$, we set $\sigma(5) = \sigma^*(5) = 1$.  We remove the run $52$
from $\pi$ and obtain
\[
  \pi = 96\;8731\qquad
  \sigma =
  \begin{smallmatrix}
    1 & 2 & 3 & 4 & 5 & 6 & 7 & 8 & 9\\
    3 & 7 & 5 & 2 & 1 & 8 & 9 &   &
  \end{smallmatrix}\qquad
  \tau = \begin{smallmatrix}
    1 & 6\\
    6 & 4
  \end{smallmatrix}.
\]
The smallest run top is now $t=8$.  There is a run to the left of
$t$, with run bottom $a = 6$.  Since $\tau(a) = 4 < 8$, we set
$\sigma(8)=4$.  We also remove the run $8731$ from $\pi$, which
leaves us with
\[
  \pi = 96\qquad
  \sigma =
  \begin{smallmatrix}
    1 & 2 & 3 & 4 & 5 & 6 & 7 & 8 & 9\\
    3 & 7 & 5 & 2 & 1 & 8 & 9 & 4 &
  \end{smallmatrix}\qquad
  \tau = \begin{smallmatrix}
    6\\
    6
  \end{smallmatrix},
\]
and it remains to set $\sigma(9) = \sigma^*(9) = 6$.

\begin{lemma}\label{l:tau-smaller}
  At any stage of the algorithm and for any run bottom~$b$ we either
  have $\tau(b) = \sigma_e^*(\rho(b))$ or $\tau(b)$ is less than or
  equal to the minimal remaining run top of $\pi$.
\end{lemma}
\begin{proof}
  Suppose that the minimal run top is $t$ and the corresponding run
  bottom is $b$.

  If $t$ is the first element of $\pi$, then, after setting
  $\sigma(t)$, only $b$ is removed from the domain of $\tau$, and no
  other values of $\tau$ are modified.

  Otherwise, let $c=\tau^{k-1}(a)$, and let
  $d=\tau(c)=\sigma^*(\rho(c))$.  Then after setting $\sigma(t)=d$,
  the value of $\tau(c)$ becomes
  \[
    \sigma^*(\rho(c)) = \sigma^*(\sigma^{-1}(d)) = \sigma^*(t) =
    \sigma^*(\rho(b)) = \tau(b).
  \]

  If $\tau(b)$ still has its original value, we have
  \[
    \tau(b) = \sigma_e^*(\rho(b)) = \sigma_e^*(t) \leq t.
  \]
  Otherwise, $\tau(b)$ is not larger than the minimal run top
  remaining when it was modified itself, which in turn must be
  smaller than $t$.
\end{proof}

\begin{lemma}
  In step~\eqref{continue-v1}, there exists an integer $k\geq 1$ such
  that $\tau^k(a) < t$.  We then also have that
  $\tau^k(a) = \sigma^*(t')$ for some run top $t' > t$.

  Thus, the algorithm is well defined and satisfies the assumption.
\end{lemma}
\begin{proof}
  Initially, $\tau$ is a permutation of the run bottoms of $\pi$.
  Thus, if $\tau^\ell(a) = (\sigma_e^*\circ \rho)^\ell(a)$ for all
  $\ell$, the cycle of $\tau$ beginning with $\tau(a)$ also contains
  $a$, which is strictly smaller than $t$.

  Otherwise, if $\tau^\ell(a) \neq (\sigma_e^*\circ \rho)^\ell(a)$
  for some $\ell$, then $\tau^\ell(a) \leq s < t$ by
  \cref{l:tau-smaller}, where $s$ is the minimal run top of what
  remained of $\pi$ when the modification of
  $\sigma^*(\rho(\tau^{\ell-1}(a)))$ occurred.

  It remains to show that $\tau^k(a) = \sigma^*(t')$ for some run top
  $t' > t$.  To this end, let $t' = \sigma^m(\tau^k(a))$ with $m$
  maximal.  Then $t' \geq t$, because all run tops strictly smaller
  than $t$ have already been assigned an image.

  Suppose that $t' = t$, that is,
  $\tau^k(a) = \sigma^*(t') = \sigma^*(t)$.  Then
  $\tau^{k-1}(a) = \rho^{-1}(t) \leq t$.  By the minimality of $k$ we
  obtain that $\tau^{k-1}(a) = t = \rho(t)$.  However, if
  $\tau(c) = \sigma^*(\rho(c)) = t$ then $\rho(c) = t$, because $t$
  has no image, which in turn implies that $t$ must be a fixed point
  of $\tau$.  This contradicts the assumption that
  $\tau^k(a) = \sigma^*(t) = t$.
\end{proof}

\begin{proposition}
  \label{p:injective}
  $\chi$ is injective, and therefore a bijection.
\end{proposition}
\begin{proof}
  Suppose that $\chi(\pi_1) = \chi(\pi_2)$.  By \cref{lem:runs},
  $\pi_1$ and $\pi_2$ have the same set of runs, which implies that
  the same $\tau$ will be used for both to compute the weak deficiencies
  of $\chi(\pi_1)$ and $\chi(\pi_2)$.  We show that the sequence of
  runs is the same in $\pi_1$ and $\pi_2$.

  Suppose that there is a run, with run top $t$ and run bottom $b$,
  which is the first run in what remains of $\pi_1$ but in what
  remains of $\pi_2$ the run bottom $a$ precedes $t$.  Then
  $\chi(\pi_1)(t) = \sigma^*(t) = \tau(b) = \tau^k(a) = \chi(\pi_2)(t)$.  Since
  $\tau$ is injective, we therefore have $\tau^{k-1}(a) = b \leq t$,
  which contradicts the minimality of $k$.

  Finally, suppose that the currently minimal run top $t$ is preceded
  by $a_1$ in what remains of $\pi_1$, and by $a_2$ in what remains
  of $\pi_2$.  Then we have
  $\chi(\pi_1)(t) = \tau^{k_1}(a_1) = \tau^{k_2}(a_2) =
  \chi(\pi_2)(t)$.  Without loss of generality, we may assume that
  $k_1 < k_2$.  Since $\tau$ is injective,
  $\tau^{k_2-k_1}(a_2) = a_1 \leq t$, contradicting the minimality of
  $k_2$.
\end{proof}

\subsection{Second variant}
\label{sec:second-variant}

We specify the final step of the algorithm as follows:
\begin{enumerate}
  \setcounter{enumi}{(3)}
\item\label{continue} Otherwise, let $i$ be the index of the run
  bottom left of $t$ in the sequence $b_1 < b_2 < \dots$ of run
  bottoms in what remains of $\pi$ and let $s_i$ be the $i$-th
  smallest element other than $s = \sigma^*(t)$ which is not yet in
  the image of $\sigma$.  Set $\sigma(t) = s_i$.
\end{enumerate}

Again, let us illustrate this procedure by resuming our example.
Initially, we have
\[
  \pi = 52\;96\;8731\;4\qquad
  \sigma =
  \begin{smallmatrix}
    1 & 2 & 3 & 4 & 5 & 6 & 7 & 8 & 9\\
    3 & 7 & 5 &   &   & 8 & 9 &   &
  \end{smallmatrix}
\]
The smallest run top is $t=4$, with corresponding run bottom $b=4$.
There is a run to the left of $t$, whose run bottom $1$ is the
smallest among the run bottoms in $\pi$ other than $b=4$.  The first
element other than $\sigma^*(t)=4$ which is not in the image of
$\sigma$ is $1$, so we set $\sigma(4)=1$.

Removing the run $4$ from $\pi$ we obtain
\[
  \pi = 52\;96\;8731\qquad
  \sigma =
  \begin{smallmatrix}
    1 & 2 & 3 & 4 & 5 & 6 & 7 & 8 & 9\\
    3 & 7 & 5 & 1 &   & 8 & 9 &   &
  \end{smallmatrix}
\]
The smallest run top is now $t=5$.  Since $t$ is the first element of
$\pi$, we set $\sigma(5) = \sigma^*(5) = 4$.

We remove the run $52$ from $\pi$, to obtain
\[
  \pi = 96\;8731\qquad
  \sigma =
  \begin{smallmatrix}
    1 & 2 & 3 & 4 & 5 & 6 & 7 & 8 & 9\\
    3 & 7 & 5 & 1 & 4 & 8 & 9 &   &
  \end{smallmatrix}
\]
The smallest run top is now $t=8$, with corresponding run bottom
$b=1$.  There is a run to the left of $t$, whose run bottom $6$ is
the smallest among the run bottoms in $\pi$ other than $b=1$.  The
first element other than $\sigma^*(t)=6$ which is not in the image of
$\sigma$ is $2$, so we set $\sigma(8)=2$.

Finally, we remove the run $8731$ from $\pi$.  This leaves us with
\[
  \pi = 96\qquad
  \sigma =
  \begin{smallmatrix}
    1 & 2 & 3 & 4 & 5 & 6 & 7 & 8 & 9\\
    3 & 7 & 5 & 1 & 4 & 8 & 9 & 2 &
  \end{smallmatrix}
\]
and it remains to set $\sigma(9) = \sigma^*(9) = 6$.

\begin{lemma}
  In step~\eqref{continue} of the algorithm, the number of run
  bottoms strictly less than $t$ equals the number of elements
  strictly less than $t$ other than $s$ which are not in the image of
  $\sigma$.

  Thus, the algorithm is well defined.  Moreover, the assumption is
  satisfied.
\end{lemma}
\begin{proof}
  We consider the effect of a single iteration of the algorithm on
  the set of run bottoms of what remains of $\pi$ and the set of
  elements which are not yet in the image of $\sigma$.  To do so,
  initialise the two sets $B$ and $S$ with the initial set of run
  bottoms $\RB$ of $\pi$.  After each iteration, remove the run
  bottom $b$ in the same run as $t$ from $B$ and remove $\sigma(t)$
  from $S$.  Thus, only elements smaller than or equal to $t$ are
  removed from $B$.

  If step~\eqref{finish} is executed, also an element smaller than or
  equal to $t$ is removed from $S$, because $\sigma^*(t) \leq t$.  On
  the other hand, suppose that in previous iterations only elements
  smaller than or equal to the then current run top were removed from
  $S$.  Then, in step~\eqref{finish}, the number of elements in
  $S\setminus s$ which are less than or equal to $t$ equals the
  number of elements in $B\setminus b$ which are less than or equal
  to $t$.  Since $b_i\in B\setminus b$, we conclude that
  $s_i\in S\setminus s$.

  Finally, since $s_i \neq \sigma^*(t)$ and all elements smaller than
  $t$ have an image, $s_i = \sigma(t')$ for some run top strictly
  larger than $t$.
\end{proof}

\begin{proposition}
  \label{p:injective2}
  $\chi$ is injective, and therefore a bijection.
\end{proposition}
\begin{proof}
  Suppose that $\chi(\pi_1) = \chi(\pi_2)$.  By \cref{lem:runs},
  $\pi_1$ and $\pi_2$ have the same set of runs.  We show that the sequence of
  runs is the same in $\pi_1$ and $\pi_2$.

  Suppose that there is a run, with run top $t$, which is the first
  run in what remains of $\pi_1$ but in what remains of $\pi_2$ the
  run bottom $b_i$ precedes $t$.  Then $s = \sigma^*(t) = s_i$, which
  is impossible.

  Otherwise, suppose that the minimal run top $t$ is preceded by
  $b_i$ in what remains of $\pi_1$, and by $b_j$ in what remains of
  $\pi_2$.  Then we have $s_i = s_j$ and therefore also $i = j$.
\end{proof}

%%%%%%%%%%%%%%%%%%%%%%%%%%%%%%%%%%%%%%%%%%%%%%%%%%%%%%%%%%%%
%
%  Bibliography
%
%%%%%%%%%%%%%%%%%%%%%%%%%%%%%%%%%%%%%%%%%%%%%%%%%%%%%%%%%%%%

\printbibliography
\end{document}